\documentclass[12pt,reqno]{amsart}
\usepackage{epsfig,listings}
\usepackage{verbatim}
\usepackage{amssymb,amsmath}
\usepackage{amsfonts,mathrsfs}
\usepackage{amsbsy}
\usepackage{graphicx}
\usepackage{subfigure}
\usepackage{caption}
\usepackage[usenames]{color} %\color{Red}
\usepackage{bm}
\usepackage{appendix}

\newtheorem{thm}{Theorem}[section]
\newtheorem{lem}[thm]{Lemma}

\theoremstyle{definition}
\newtheorem{defn}[thm]{Definition}

\newtheorem{rem}[thm]{Remark}

\usepackage{color}
\usepackage{setspace}

\usepackage [colorlinks, citecolor=blue, anchorcolor=blue, linkcolor=blue ]{hyperref}%%
\usepackage{framed} %
\usepackage{diagbox}%
\usepackage[justification=centering]{caption}

\begin{document}
	\title{ Embedding hypercubes into torus and
		Cartesian product of paths and cycles 
		for minimizing wirelength }

	\author{Zhiyi Tang}
	\address{
	School of Mathematics and Physics, Hubei Polytechnic University, Huangshi 435003, PR China.}
	\email{tangzhiyi@hbpu.edu.cn}
	
\begin{abstract}

Though embedding problems have been considered for several regular graphs\cite{BCHRS1998,AS2015,ALDS2021},
%such as hypercube into grid\cite{PM2009}, %binary tree into grid\cite{OS2000}, 
%honeycomb networks into hypercubes\cite{DWNS04}, et al, 
it is still an open problem for hypercube into torus\cite{PM2011,AS2015}.
In the paper, we prove the conjecture mathematically and 
obtain the minimum wirelength of embedding for hypercube into Cartesian product of paths and/or cycles.
In addition, we explain that Gray code embedding  is an optimal strategy in such embedding problems.
		
		
\end{abstract}
	
\maketitle	
	
\section{Introduction}

%Graph embedding is important in parrallel algorithm, parrallel computer, or multiprocessor systems.
Task mapping in modern high performance parallel computers can be modeled as a graph embedding problem.
Let $G(V,E)$ be a simple and connected graph with vertex set $V(G)$ and edge set $E(G)$.
Graph embedding\cite{BCHRS1998,AS2015,ALDS2021} is an
ordered pair $<f,P_f>$ of injective mapping between the guest graph $G$ and  the host graph $H$ such that
\begin{itemize}
	\item[(i)] $f:V(G)\rightarrow V(H)$, and
	\item[(ii)] $P_f: E(G)\rightarrow$
	$\{P_f(u,v):$ $P_f(u,v)$ is a path in $H$ between $f(u)$ and $f(v)$ for $\{u,v\}\in E(G)\}$.
\end{itemize}

It is known that the topology mapping problem is NP-complete\cite{HS2011}.
Since Harper\cite{H1964} in 1964 and Bernstein\cite{Bernstein1967} in 1967, a series of  embedding problems have been studied\cite{E1991,OS2000,DWNS04,FJ2007,LSAD2021}.
The quality of an embedding can be measured by certain cost criteria.
One of these criteria is the wirelength.
Let $WL(G,H;f)$ denote the wirelength of $G$ into $H$ under the embedding $f$.
Taking over all embeddings $f$, the minimum wirelength of $G$ into $H$ is defined as
$$WL(G,H)=\min\limits_{f} WL(G,H;f).$$

Hypercube is one of the most popular, versatile and efficient topological structures of interconnection networks\cite{Xu2001}. 
More and more studies related to hypbercubes have been performed\cite{Chen1988,PM2009,PM2011,RARM2012}.
Manuel\cite{PM2011} et al. computated the minimum wirelength of embedding hypercube into a simple cylinder. In that paper, the wirelenth for hypercube into general cylinder and torus were given as  conjectures. 
Though Rajan et al.\cite{RRPR2014} and Arockiaraj et al.\cite{AS2015} studied those embedding problems, the two conjectures are still open.
We recently gave  rigorous proofs of  hypercubes into cycles\cite{LT2021} and cylinders (the first conjecture)\cite{Tang2022} successively. 
Using those techniques and process, %followed in the previous works, 
we try to settle the last conjecture for torus. 
In the paper, we also generaliz the results to other Cartesian product of  paths and/or cycles.

It is seen that the grid, cylinder and torus are Cartesian product of graphs. In the past, the vertices of those graphs are labeled by a series of nature numbers\cite{PM2009,PM2011,AS2015,Tang2022}. 
But it is not convenient for some higher dimensional graphs.
To describe a certain embedding efficiently,
we apply tuples to lable the vertices in the paper.
By the tool of Edge Isoperimetric Problem(EIP)\cite{H2004}, we estimate and explain the minimal wirelength for hypercube into 
torus and other Cartesian product of graphs. 

\noindent\textbf{Notation.} 
For $n\ge 1$, 
we define $Q_n$ to be the hypercube with vertex-set $\{0,1\}^n$, 
where two $0-1$ vectors are adjacent if they differ in exactly one coordinate \cite{RR2022}.

\noindent\textbf{Notation.}
 An $r_1\times r_2$ grid with $r_1$ rows and $r_2$ colums is represented by $P_{r_1}\times P_{r_2}$ where the rows are labeled $1,2,\ldots, r_1$ and 
the columns are labeled $1,2,\ldots, r_2$ \cite{PM2009}.
The torus $C_{r_1}\times C_{r_2}$ is a $P_{r_1}\times P_{r_2}$ with a wraparound edge in each column and a wrapround edge in each row.
	
\textbf{Main Results}
\begin{thm}\label{ccthm}
	For any  $n_1,\ n_2\ge2,\ n_1+ n_2=n$.
	The minimum wirelength  of hypercubes into torus is
		\begin{equation*}\label{cpwl}
	WL(Q_n,C_{2^{n_1}}\times C_{2^{n_2}})=
	2^{n_2}(3\cdot 2^{2n_1-3}-2^{n_1-1})+
	2^{n_1}(3\cdot 2^{2n_2-3}-2^{n_2-1}).
		\end{equation*}	
	Moreover, Gray code embedding is an optimal embedding.	
\end{thm}

\noindent\textbf{Notation.} 
Cartesian product of paths and/or cycles  is denoted by 
$\mathscr{G}=\mathscr{G}_1\times \mathscr{G}_2\times \cdots \times \mathscr{G}_k,$
where $\mathscr{G}_i \in \{P_{2^{n_i}}, C_{2^{n_i}}\},  1\le i \le k$.	

\begin{thm}\label{carthm}
For any $k>0$, $ n_i\ge2$, and  $\sum_{i=1}^{k}n_i=n$.
The minimum wirelength  of hypercubes into Cartesian product $\mathscr{G}$ is
\begin{equation*}
	WL(Q_n,\mathscr{G})=\sum_{i=1}^{k}\mathscr{L}_i,
\end{equation*}
where	\begin{equation*}
	\mathscr{L}_i=\left\{
	\begin{array}{rcl}
		&2^{n-n_i}(3\cdot 2^{2n_i-3}-2^{n_i-1}),&\mbox{if}\ \ \mathscr{G}_i=C_{2^{n_i}},\\
		&2^{n-n_i}(2^{2n_i-1}-2^{n_i-1}),&\mbox{if}\ \ \mathscr{G}_i=P_{2^{n_i}}.
	\end{array}
	\right.
\end{equation*}
Moreover, Gray code embedding is an optimal embedding.
\end{thm}
	
The paper is organized as follows. 
In Section \ref{Preliminaries}, some definitions and elementary properties are introduced. 
In Section \ref{sec:torus}, we explain the Gray code embedding is an optimal strategy for hypercube into torus. 
Section \ref{sec:cartesian} is devoted to  Cartesian products of paths and/or cycles. 
%Section 5 contains a summary.

\section{Preliminaries}	\label{Preliminaries}
EIP has been used as  a powful tool in the computation of minimum wirelength of graph embedding\cite{H2004}.
 EIP is to determine a subset $S$ of cardinality $k$  of a  graph $G$ such that the edge cut separating this subset from its complement has minimal size. Mathematically, Harper denotes
	$$\Theta(S)=\{ \{u,v\}\in E(G)\ : u\in S, v\notin S  \}.$$ 
%	Generally, for a given $k$, it is difficult to find a subset $S$ such that $|\Theta(S)|$ is the minimal.
	%Harper solved EIP when $G$ is an $n-$dimensional hypercube  $Q_n$.
	For any $S\subset V(Q_n)$, use $\Theta(n,S)$ in place of $\Theta(S)$ and let $|\Theta(n,S)|$ be $\theta(n,S)$. 
\begin{lem}\label{swap}
	Take a subcube $Q_{n_1}$ of $Q_n$, and $S_1\subset V(Q_{n_1}), S_2\subset V(Q_{n-n_1})$, then
	\begin{equation*}
		\theta(n,S_1\times S_2)=\theta(n,S_2 \times S_1).
	\end{equation*}
\end{lem}	
\begin{proof}
	By the definition of hypercube $Q_n$, 
	there is an edge  connected in $S_1\times S_2$ if and only if
	there is an edge connected in $S_2\times S_1$.
\end{proof}

The following lemma is efficient technique to find the exact wirelength. 

\begin{lem}\cite{Tang2022}\label{wl}
	Let $f$ be an embedding of $Q_n$ into $H$. 
	Let $(L_i)_{i=1}^{m}$ be a partition of $E(H)$.  
	For each $1\le i \le m$, $(L_i)_{i=1}^{m}$ satisfies:
	\begin{itemize}
		\item[\bf{\normalfont (A1)}]$L_i$ is an edge cut of $H$ such that $L_i$ disconnects $H$ into two components and one of induduced vertex sets is denoted by $l_i$;
		\item[\bf{\normalfont (A2)}]$|P_f(u,v)\cap L_i|$ is one if $\{u,v\}\in \Theta(n,f^{-1}(l_i))$ and zero otherwise for any $\{u,v\}\in E(Q_n)$.
	\end{itemize}
	Then $$WL(Q_n,H;f)=\sum_{i=1}^{m}\theta(n,f^{-1}(l_i)).$$
\end{lem}	

\noindent\textbf{Notation.} 
$N_n=\{1,2,\cdots,n\}$, and $F_i^{n}=\{i,i+1,\cdots,i+2^{n-1}-1\}, \quad 1\le i \le 2^{n-1}$. 

\noindent\textbf{Notation.} 
Let $(i,j)$ denote a vertice in row $i$ and column $j$ of  cylinder $C_{2^{n_1}}\times P_{2^{n_2}}$, where $1\le i \le 2^{n_1}$ and $1\le j \le 2^{n_2}$.
Then $V(C_{2^{n_1}}\times P_{2^{n_2}})=N_{2^{n_1}}\times N_{2^{n_2}}.$
It is seen that the vertex sets 
$F_i^{n_1}\times N_{2^{n_2}}=\{(x_1,x_2):x_1\in F_i^{n_1},\ x_2\in N_{2^{n_2}} \}$ and
$N_{2^{n_1}}\times N_{j}=\{(x_1,x_2):  x_1\in N_{2^{n_1}},\ x_2\in N_{j} \}$ are equivalent to 
$A_i$ and $B_j$ defined in \cite{Tang2022} , respectively.
See Fig.\ref{label_1} and Fig.\ref{label_2} for  examples.

\begin{figure}[htbp]
	\centering
	\includegraphics[width=5in]{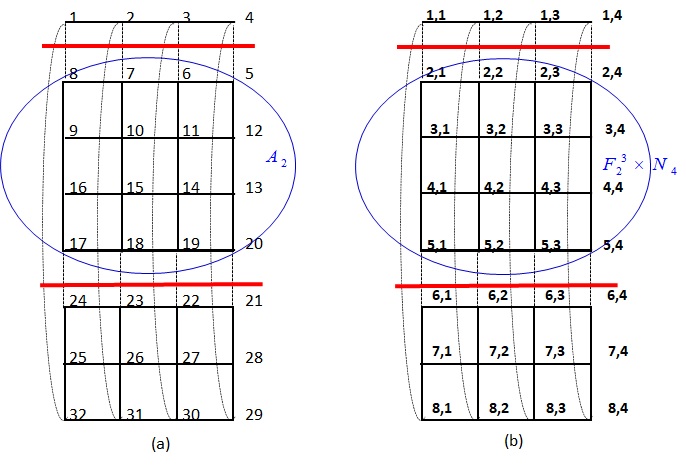}
	\caption{$A_2$ and $F_2^{3}\times N_{2^{2}}$ in cylinder $C_{2^{3}}\times P_{2^{2}}$ }
	\label{label_1}
\end{figure}

\begin{figure}[htbp]
	\centering
	\includegraphics[width=5in]{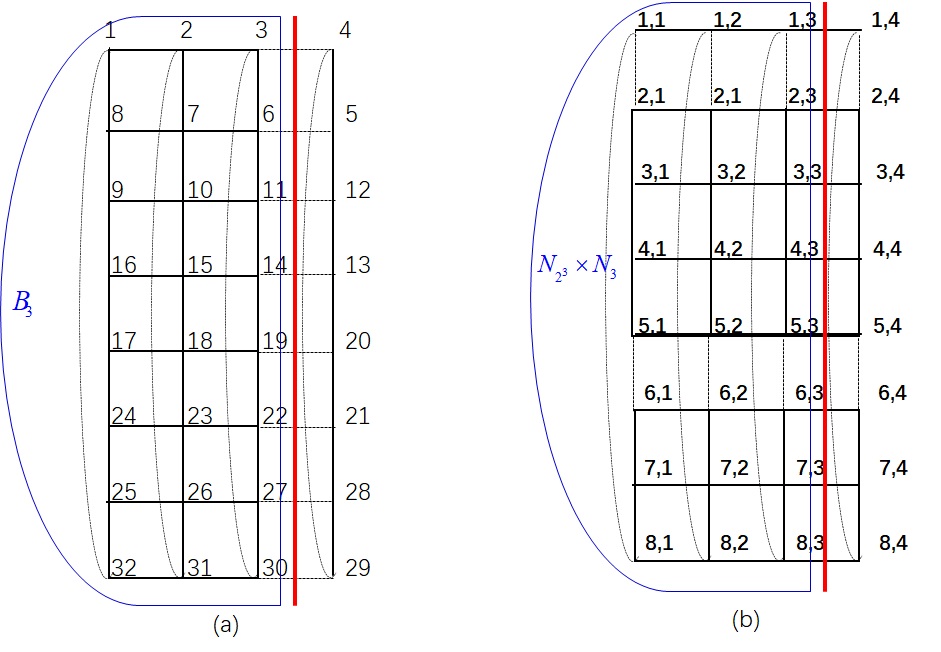}
	\caption{$B_3$ and $N_{2^3}\times N_{3}$ in cylinder $C_{2^{3}}\times P_{2^{2}}$}
	\label{label_2}
\end{figure}

Now we generalize  Gray code map $\xi_n:\{0,1\}^n\rightarrow \{1,2,\cdots,2^n\}$ defined in \cite{LT2021,Tang2022}. 
Define $k$-order Gray code map $\xi_{n_1\ldots n_k}$ corresponding to $k$ components.

\begin{defn}
$k$-order Gray code map	$\xi_{n_1\ldots n_k}$ is given by $\xi_{n_1\ldots n_k}:\{0,1\}^n\rightarrow N_{2^{n_1}}\times\cdots \times N_{2^{n_k}}$,\ i.e.,
	$$\xi_{n_1\ldots n_k}(v)=\xi_{n_1\ldots n_k}(v_1\ldots v_k)=(\xi_{n_1}(v_1),\ldots,\xi_{n_k}(v_k)),$$
	where $n_1+\ldots+n_k=n$, and $v=v_1\ldots v_k\in \{0,1\}^n, v_i\in \{0,1\}^{n_i}, 1\le i \le k$.
\end{defn}
For example, $\xi_{32}(11011)=(\xi_{3}(110),\xi_{2}(11))=(5,3)$.

According to the rule of Gray code map, we have that
\begin{equation*}\label{change}
	\xi_{n_1n_2}^{-1}(F_i^{n_1}\times N_{2^{n_2}})=\xi_n^{-1}(A_i),\ \ \xi_{n_1n_2}^{-1}(N_{2^{n_1}}\times N_{j})=\xi_n^{-1}(B_j).
\end{equation*}

Together with (12) and (13) in \cite{Tang2022}, we have that
\begin{subequations}\label{cpwl1+2}
	\begin{eqnarray}
		&&\sum_{i=1}^{2^{n_1-1}} \theta(n,\xi_{n_1n_2}^{-1}(F_i^{n_1}\times N_{2^{n_2}}))=
		2^{n-n_1}(3\cdot 2^{2n_1-3}-2^{n_1-1}).\\
		&&\sum_{j=1}^{2^{n_2}-1}\theta(n,\xi_{n_1n_2}^{-1}(N_{2^{n_1}}\times N_{j}))
		=2^{n-n_2}(2^{2n_2-1}-2^{n_2-1}).
	\end{eqnarray}
\end{subequations}

Let $f: \{0,1\}^n\rightarrow N_{2^{n_1}}\times N_{2^{n_2}}$ be an
embedding of $Q_n$ into $C_{2^{n_1}}\times P_{2^{n_2}}$. 
Theorems 5.2 and 5.1 in \cite{Tang2022} is rewritten as 
\begin{subequations}\label{cp1+2}
	\begin{eqnarray}
		&&\sum_{i=1}^{2^{n_1-1}} \theta(n,f^{-1}(F_i^{n_1}\times N_{2^{n_2}}))\ge
		\sum_{i=1}^{2^{n_1-1}} \theta(n,\xi_{n_1n_2}^{-1}(F_i^{n_1}\times N_{2^{n_2}})).\\
		&&\sum_{j=1}^{2^{n_2}-1}\theta(n,f^{-1}(N_{2^{n_1}}\times N_{j}))\ge
		\sum_{j=1}^{2^{n_2}-1} \theta(n,\xi_{n_1n_2}^{-1}(N_{2^{n_1}}\times N_{j})).
	\end{eqnarray}
\end{subequations}

 Cylinder $C_{2^{n_1}}\times P_{2^{n_2}}$ can also be observed as $P_{2^{n_2}}\times C_{2^{n_1}}$.
 Let $f: \{0,1\}^n\rightarrow N_{2^{n_2}}\times N_{2^{n_1}}$ be an
 embedding of $Q_n$ into $P_{2^{n_2}}\times C_{2^{n_1}}$, then \eqref{cp1+2}
is rewritten as
\begin{subequations}\label{change:cp1+2}
	\begin{eqnarray}
		&&\sum_{i=1}^{2^{n_1-1}} \theta(n,f^{-1}(N_{2^{n_2}}\times F_i^{n_1}))\ge
		\sum_{i=1}^{2^{n_1-1}} \theta(n,\xi_{n_2n_1}^{-1}(N_{2^{n_2}}\times F_i^{n_1})).\\
		&&\sum_{j=1}^{2^{n_2}-1}\theta(n,f^{-1}(N_j\times N_{2^{n_1}}))\ge
		\sum_{j=1}^{2^{n_2}-1} \theta(n,\xi_{n_2n_1}^{-1}(N_j\times N_{2^{n_1}})).
	\end{eqnarray}
\end{subequations}

\begin{rem}
	It is seen  that
$\xi_{n_1n_2}^{-1}(F_i^{n_1}\times N_{2^{n_2}}) =\xi_{n_1}^{-1}(F_i^{n_1})\times V(Q_{{n_2}})$. Then, by Lemma \ref{swap}, we get that $\theta(n,\xi_{n_1n_2}^{-1}(F_i^{n_1}\times N_{2^{n_2}}))=\theta(n,\xi_{n_2n_1}^{-1}(N_{2^{n_2}}\times F_i^{n_1}))$.
\end{rem}

\section{hypbercubes into torus}\label{sec:torus}

In this section, we prove Theorem \ref{ccthm} in the following procedures.

\noindent$\bullet$ \textbf{Labeling.}\ \ 
Let a  binary tuple set denote the vertex set of torus $C_{2^{n_1}}\times C_{2^{n_2}}$ , that is
$$V(C_{2^{n_1}}\times C_{2^{n_2}})=
\{x=(x_1,x_2): 1\le x_i\le 2^{n_i},\ i=1,2\}=N_{2^{n_1}}\times N_{2^{n_2}}.$$ 
The edge set 
$E(C_{2^{n_1}}\times C_{2^{n_2}})$ is composed of $\mathscr{E}_1$ and $\cup \mathscr{E}_2$, 
where
$$\begin{array}{rcl}
	\mathscr{E}_1&=&\{\{(x_1,x_2),(x_1',x_2)\}:\{x_1,x_1'\}\in E(C_{2^{n_1}}), x_2\in  N_{2^{n_2}}\},\\
	\mathscr{E}_2&=&\{\{(x_1,x_2),(x_1,x_2')\}: x_1\in  N_{2^{n_1}}, \{x_2,x_2'\}\in E(C_{2^{n_2}})\}.
\end{array}$$

\noindent$\bullet$ \textbf{Partition.}\ \ Construct a partition of the edge set of torus.

\textbf{Step 1.}\ 
%There are $2^{n_2}$ cycles in cylinder $C_{2^{n_1}}\times C_{2^{n_2}}$.
%Construct a partition $(\mathscr{X}_{ij})_{j=1}^{2^{n_i-1}}$ of the edge set of  cycle $C_{2^{n_i}}$, $i=1,2$. 
For  each $i=1,2$, $j=1,\ldots,2^{n_i-1}$,  
let $\mathscr{X}_{ij}$ be an edge cut of the cycle $C_{2^{n_i}}$ such that  $\mathscr{X}_{ij}$ disconnects $C_{2^{n_i}}$ into two components where the induced vertex set
is $F_j^{n_i}$.

\textbf{Step 2.}\
 For $i=1,2$, denote
\begin{equation*}
	\mathscr{P}_{ij}=\bigcup_{\{x_i,x_i'\} \in \mathscr{X}_{ij}}\{\{x,x'\}\in \mathscr{E}_i\},
\end{equation*}
then $\{\mathscr{P}_{ij}:1\le i \le 2, 1\le j \le 2^{n_i-1}\}$ is the partition of $E(C_{2^{n_1}}\times C_{2^{n_2}})$.

\noindent$\bullet$ \textbf{Computation.}\ \
 Notice that for each $i,j$, 
 $\mathscr{P}_{ij}$ is an edge cut of the  torus $C_{2^{n_1}}\times C_{2^{n_2}}$.
 $\mathscr{P}_{1j}$ disconnects the torus into two components where the induced vertex set is $F_j^{n_1}\times N_{2^{n_2}}$, and
 $\mathscr{P}_{2j}$ induces vertex set $N_{2^{n_1}}\times F_j^{n_2}$.
 See Fig.\ref{label_3} for an example.

\begin{figure}[htbp]
	\centering
	\includegraphics[width=5in]{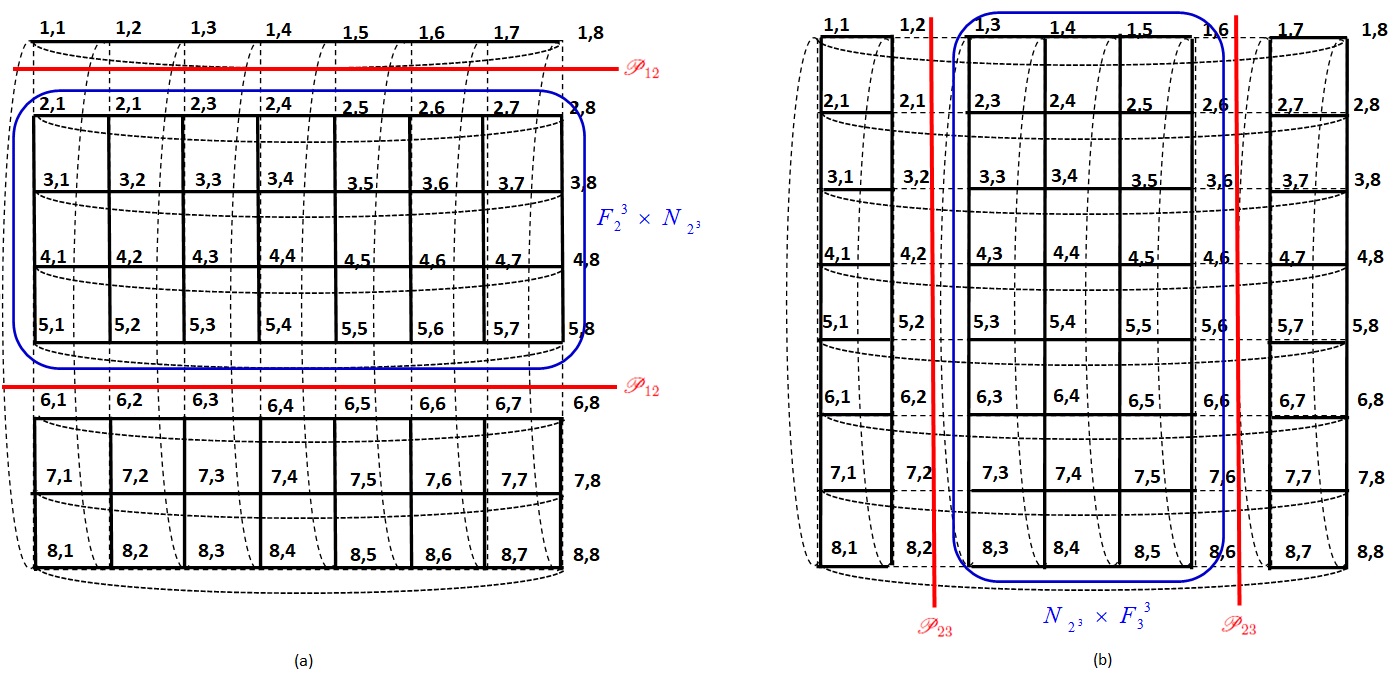}
	\caption{(a) Edge cut $\mathscr{P}_{12}$ disconnects $C_{2^3}\times C_{2^3}$ into two components,where the induced vertex set is $F_2^{3}\times N_{2^{3}}$.
		(b) Edge cut $\mathscr{P}_{23}$ disconnects $C_{2^3}\times C_{2^3}$ into two components,where the induced vertex set is $N_{2^{3}}\times F_3^{3}$.
	}
	\label{label_3}
\end{figure}

 Let $f: \{0,1\}^n\rightarrow N_{2^{n_1}}\times N_{2^{n_2}}$ be an
 embedding of $Q_n$ into $C_{2^{n_1}}\times C_{2^{n_2}}$.
 Under the partition $\{\mathscr{P}_{ij}:1\le i \le 2, 1\le j \le 2^{n_i-1}\}$  and Lemma \ref{wl},
 the wirelength is written as a summation related to function $\theta$, i.e.,
 \begin{equation}\label{ccsum}
 WL(Q_n,C_{2^{n_1}}\times C_{2^{n_2}};f)=
 \sum_{j=1}^{2^{n_1-1}}\theta(n,f^{-1}(F_j^{n_1}\times N_{2^{n_2}}))+
 \sum_{j=1}^{2^{n_2-1}}\theta(n,f^{-1}(N_{2^{n_1}}\times F_j^{n_2})).
 \end{equation}

According to Lemma \ref{swap} and (\ref{cpwl1+2}a), we have that
\begin{equation}\label{cc1}
	\sum_{j=1}^{2^{n_2-1}} \theta(n,\xi_{n_1n_2}^{-1}(N_{2^{n_1}}\times F_j^{n_2}))=
	2^{n-n_2}(3\cdot 2^{2n_2-3}-2^{n_2-1}).
\end{equation}

According to Lemma \ref{swap} and (\ref{change:cp1+2}a), we have that
\begin{equation}\label{cc2}
	\sum_{j=1}^{2^{n_2-1}} \theta(n,f^{-1}(N_{2^{n_1}}\times F_j^{n_2}))\ge
	\sum_{j=1}^{2^{n_2-1}} \theta(n,\xi_{n_1n_2}^{-1}(N_{2^{n_1}}\times F_j^{n_2})).
\end{equation}
%Together with \eqref{ccsum},(\ref{cpwl1+2}a),\eqref{cc1}, (\ref{cp1+2}a),and \eqref{cc2},
Combining above three fomulas and (\ref{cpwl1+2}a),(\ref{cp1+2}a),
 Theorem \ref{ccthm} holds.

\section{hypercubes into Cartesian product of paths and/or cycles}
\label{sec:cartesian}

In this section, we prove Theorem \ref{carthm} in three parts.
The first part follows the analogous process as Section \ref{sec:torus}.
Then we obtain the wirelength under Gray code embedding.
In the end, we conclude that  Gray code embedding is an optimal strategy.

\subsection{Compuation of embedding wirelength}\label{sub1}\

\noindent$\bullet$ \textbf{Labeling.}\ \ Let $$V(\mathscr{G})=\{x=(x_1,\ldots,x_k):x_i\in N_{2^{n_i}}, 1\le i\le k \}
=N_{2^{n_1}}\times \cdots \times N_{2^{n_k}}$$ 
be the vertex set of  Cartesian product $\mathscr{G}$ of $k$ paths and/or cycles.
The edge set $E(\mathscr{G})$ of Cartesian product $\mathscr{G}$ is composed of all edges $\mathscr{E}_i$ correspongding to $k$ paths and/or cycles, denoted by
$E(\mathscr{G})=\bigcup_{i=1}^{k}\mathscr{E}_i$.

\noindent$\bullet$ \textbf{Partition.}\ \ Construct a partition  of the edge set of Cartesian product $\mathscr{G}$.

\textbf{Step 1.}\
For  each $i=1,\ldots,k$, $j=1,\ldots,2^{n_i-1}$, $\mathscr{X}_{ij}$ is  described earlier in Section \ref{sec:torus}.
For  each $i=1,\ldots,k$, $j=1,\ldots,2^{n_i}-1$,  
let $\mathscr{Y}_{ij}$ be an edge cut of the path $P_{2^{n_i}}$ such that  $\mathscr{Y}_{ij}$ disconnects $P_{2^{n_i}}$ into two components where the induced vertex set is $N_j$.

\textbf{Notation.}\ For $1\le i\le k$, let $q_i$ be $2^{n_i-1}$ if  $\mathscr{G}_i=C_{2^{n_i}}$ and $2^{n_i}-1$ otherwise $\mathscr{G}_i=P_{2^{n_i}}$. 
For $j=1,\ldots,q_i$, denote

\begin{equation*}\label{huaF}
	\mathscr{F}_{ij}=\left\{
	\begin{array}{cl}
		\mathscr{X}_{ij},&\mbox{if}\quad \mathscr{G}_i=C_{2^{n_i}},\\
		\mathscr{Y}_{ij},&\mbox{if}\quad \mathscr{G}_i=P_{2^{n_i}}.
	\end{array}
	\right.
\end{equation*}

\textbf{Step 2.}\
For $i=1,\ldots,k$, $j=1,\ldots,q_i$, denote
\begin{equation*}
\mathscr{P}_{ij}=\bigcup_{\{x_i,x_i'\} \in \mathscr{F}_{ij}}\{\{x,x'\}\in \mathscr{E}_i\},
\end{equation*}
then $\{\mathscr{P}_{ij}:1\le i \le k, 1\le j \le q_i\}$ is a partition of $E(\mathscr{G})$.

\noindent$\bullet$ \textbf{Computation.}\ \
 Notice that for each $i,j$, 
$\mathscr{P}_{ij}$ is an edge cut of Cartesian product  $\mathscr{G}$.
Define a vertext set $\mathscr{A}_{ij}$ to be $F_j^{n_i}$  if $\mathscr{G}_i=C_{2^{n_i}}$ and $N_j$ otherwise $\mathscr{G}_i=P_{2^{n_i}}$.

\textbf{Notation.}\
\begin{equation}\label{Bij}
	\begin{array}{cl}
		&\mathscr{B}_{1j}=\mathscr{A}_{1j}\times N_{2^{n_2}}\times \cdots \times N_{2^{n_k}},\ \
		\mathscr{B}_{kj}=N_{2^{n_1}}\times \cdots \times N_{2^{n_{k-1}}}\times \mathscr{A}_{kj},\\
		&\mathscr{B}_{ij}=N_{2^{n_1}}\times \cdots \times \mathscr{A}_{ij}\times \cdots \times N_{2^{n_k}}, \ \ 1<i< k.
	\end{array}
\end{equation}

It is seen that $\mathscr{P}_{ij}$ disconnects $\mathscr{G}$ into two components where the induced vertex set $\mathscr{B}_{ij}$.
Let $f: \{0,1\}^n\rightarrow N_{2^{n_1}}\times \cdots \times N_{2^{n_k}}$ be an embedding of $Q_n$ into $\mathscr{G}$.
Under the partition $\{\mathscr{P}_{ij}:1\le i \le k, 1\le j \le q_i\}$  and Lemma \ref{wl},
the wirelength is written as a summation related to function $\theta$, i.e.,
\begin{equation}\label{wlG}
	WL(Q_n,\mathscr{G};f)=\sum_{i=1}^{k}\sum_{j=1}^{q_i}\theta(n,f^{-1}(\mathscr{B}_{ij})).
\end{equation}

\subsection{The wirelength under Gray code embedding}\label{sub2}\ \

We deal with the wirelength under Gray code embedding in two cases:
one is that $\mathscr{G}_i$ is  cycle $C_{2^{n_i}}$, 
and the other is that $\mathscr{G}_i$ is  path $P_{2^{n_i}}$.
In the following, set $1\le i \le k, 1\le j \le q_i$.

\begin{lem}\label{BWL1}\
If $\mathscr{G}_i$ is  cycle $C_{2^{n_i}}$, then we have that	
\begin{equation*}
	\sum_{j=1}^{q_i}\theta(n,\xi_{n_1\ldots n_k}^{-1}(\mathscr{B}_{ij}))=
	2^{n-n_i}(3\cdot 2^{2n_i-3}-2^{n_i-1}).
\end{equation*}
\end{lem}
\begin{proof} By the Notation \eqref{Bij}, we have that
	\begin{equation*}
		\begin{array}{rcl}
			\xi_{n_1\ldots n_k}^{-1}(\mathscr{B}_{ij})&=&
			\xi_{n_1\ldots n_k}^{-1}(N_{2^{n_1}}\times \cdots \times F_j^{n_i}\times \ldots \times N_{2^{n_k}})\\
			&=&V(Q_{n_1})\times \ldots\times\xi_{n_i}^{-1}(F_j^{n_i})
			\times\ldots\times V(Q_{n_k})\\
			&=&V(Q_{n_1+\ldots+n_{i-1}})\times
			\xi_{n_i}^{-1}(F_j^{n_i})\times V(Q_{n_{i+1}+\ldots+n_k}).
		\end{array}
	\end{equation*}
Moreover, by Lemma \ref{swap}, we have that 
	\begin{equation*}
	\begin{array}{rcl}
		&&\theta(n,V(Q_{n_1+\ldots+n_{i-1}})\times
		\xi_{n_i}^{-1}(F_j^{n_i})\times V(Q_{n_{i+1}+\ldots+n_k}))\\
		&=&\theta(n,\xi_{n_i}^{-1}(F_j^{n_i})\times 
		V(Q_{n_{i+1}+\ldots+n_k})\times V(Q_{n_1+\ldots+n_{i-1}}))\\
		&=&\theta(n,\xi_{n_i}^{-1}(F_j^{n_i})\times V(Q_{n-n_i})).
	\end{array}
\end{equation*}

Therefore, Lemma \ref{BWL1} follows from (\ref{cpwl1+2}a).
\end{proof}

Similarly, we write the following lemma.
\begin{lem}\label{BWL2}
If $\mathscr{G}_i$ is path $P_{2^{n_i}}$, then we have that
	\begin{equation*}
		\sum_{j=1}^{q_i}\theta(n,\xi_{n_1\ldots n_k}^{-1}(\mathscr{B}_{ij}))=
		2^{n-n_i}(2^{2n_i-1}-2^{n_i-1}).
	\end{equation*}
\end{lem}

Combining \eqref{wlG}, Lemma \ref{BWL1} and Lemma \ref{BWL2}, we get the wirelength under Gray code embedding of hypercube into Cartesian product of paths and/or cycles. That is
\begin{equation*}\label{wlgray}
	WL(Q_n,\mathscr{G};\xi_{n_1\ldots n_k})=\sum_{i=1}^{k}\mathscr{L}_i,
\end{equation*}
where $\mathscr{L}_i$ is defined in Theorem \ref{carthm}.

\subsection{Minimum wirelength}\label{sub3}\

We show that
Gray code embedding wirelength is the lower bound of wirelength for hypercube into Cartesian product of paths and/or cycles. 
According to \eqref{wlG}, it is sufficient to prove that

\begin{lem}\label{finalieq}
	Let $f: \{0,1\}^n\rightarrow N_{2^{n_1}}\times \cdots \times N_{2^{n_k}}$ be an embedding of $Q_n$ into $\mathscr{G}$, then 
	\begin{equation*}
		\sum_{i=1}^{k}\sum_{j=1}^{q_i}\theta(n,f^{-1}(\mathscr{B}_{ij}))\ge
		\sum_{i=1}^{k}\sum_{j=1}^{q_i}\theta(n,\xi_{n_1\cdots n_k}^{-1}(\mathscr{B}_{ij})).
	\end{equation*}
\end{lem}

\begin{proof}
	To prove this lemma, we only consider  that $i=1$, since a
	similar argument works for the other  $2\le i\le k$.
	
\noindent$\bullet$ \textbf{Case 1.}\ \
	$\mathscr{G}_1=C_{2^{n_1}}$. 
	
For $1\le j \le q_1=2^{n_1-1}$,
$f^{-1}(\mathscr{B}_{1j})=f^{-1}(F_j^{n_1}\times N_{2^{n_2}}\times \cdots \times N_{2^{n_k}})$.
Define a bijective map $f_1$ from $N_{2^{n_1}}\times N_{2^{n_2}}\times \cdots \times N_{2^{n_k}}$ to $N_{2^{n_1}}\times N_{2^{n-n_1}}$, where
	\begin{equation*}\label{f1}
		f_1(x_1,x_2,\cdots,x_k)=(x_1, x_k+\sum_{i=2}^{k-1}(x_i-1)2^{\sum_{j=i+1}^{k}n_j}).
	\end{equation*}
It is clear that $f_1(F_j^{n_1}\times N_{2^{n_2}}\times \cdots \times N_{2^{n_k}})=F_j^{n_1}\times N_{2^{n-n_1}}$.
Moreover, we have that
\begin{equation*}
	f^{-1}(\mathscr{B}_{1j})=f^{-1}\circ f_1^{-1}(F_j^{n_1}\times N_{2^{n-n_1}})=
	(f_1\circ f)^{-1}(F_j^{n_1}\times N_{2^{n-n_1}}).
\end{equation*}
Notice that $f_1\circ f$ is an arbitrary map from $\{0,1\}^n$ to $N_{2^{n_1}}\times N_{2^{n-n_1}}$, then, by (\ref{cp1+2}a), we have that
$\sum_{i=1}^{2^{n_1}-1} \theta(n,f^{-1}(\mathscr{B}_{1j}))\ge
\sum_{i=1}^{2^{n_1}-1} \theta(n,\xi_{n_1}^{-1}(F_j^{n_1})\times V(Q_{n-n_1}))$. Therefore, we conclude that
\begin{equation}\label{B1p}
\sum_{i=1}^{2^{n_1}-1} \theta(n,f^{-1}(\mathscr{B}_{1j}))\ge
	\sum_{i=1}^{2^{n_1}-1} \theta(n,\xi_{n_1\cdots n_k}^{-1}(\mathscr{B}_{1j})).
\end{equation}

\noindent$\bullet$ \textbf{Case 2.}\ \
$\mathscr{G}_1=P_{2^{n_1}}$. 
By a similar analysis, we also get \eqref{B1p}.

Combining \textbf{Case 1} and \textbf{Case 2}, the case for $i=1$ is proved. Thus the lemma holds.
\end{proof}

\noindent \textbf{Proof of Theorem \ref{carthm}.}  Theorem \ref{carthm} follows from Subsection \ref{sub1} to \ref{sub3}.

\noindent\textbf{Acknowledgements}  
The author is grateful to Prof. Qinghui Liu for his thorough review and suggestions.
%The author also thanks the anonymous referees for useful comments.
This work is supported by the National Natural Science Foundation of China, No.11871098.

%{\color{red}$\mathscr{P}_{23}$}\ \ {\color{red}$\mathscr{P}_{12}$}


\begin{thebibliography}{90}	
	
%definition of embedding
\bibitem{BCHRS1998}
S. L. Bezrukov, J. D. Chavez, L. H. Harper, M. R\"ottger, U. P. Schroeder, {Embedding of hypercubes into grids}, In: Brim L., Gruska J., Zlatuka J. (eds) Mathematical Foundations of Computer Science 1998, Lecture Notes in Computer Science, vol 1450, Springer, Berlin, Heidelberg,
693–701.
	
\bibitem{AS2015}
M. Arockiaraj, A. J. Shalini, Conjectures on wirelength of hypercube into cylinder and torus, Theoretical Computer Science, 595(2015), 168–171.
	
\bibitem{ALDS2021}
Micheal Arockiaraj, Jia-Bao Liu, J. Nancy Delaila and Arul Jeya Shalini, {On the optimal layout of balanced complete multipartite graphs into grids and tree related structures},
Discrete Applied Mathematics, 288(2021), 50–65.	

%abstract
\bibitem{PM2011}
P. Manuel, M. Arockiaraj, I. Rajasingh, B. Rajan,  {Embedding hypercubes into cylinders,snakes and caterpillars for minimizing wirelength},
Discrete Applied Mathematics, 159(2011), 2109–2116.

%NP
\bibitem{HS2011}
T. Hoefler and M. Snir, {Generic topology mapping strategies for large-scale parallel architectures}, ICS ’11: Proceedings of the international conference on Supercomputing,  May(2011), 75-84.
http://doi.acm.org/10.1145/1995896.1995909.

%since study
\bibitem{H1964}
L. H. Harper, Optimal assignments of numbers to vertices, J. SIAM 12:1 (1964), 131–135.	

\bibitem{Bernstein1967}%%
A. J. Bernstein,  {Maximality connected arrays on the $n$-cube}, SIAM J.Appl. Math., 15:6(1967), 1485–1489.

%embedding work%%%%%%%%%%%%\cite{E1991,OS2000,DWNS04,FJ2007,LSAD2021}
\bibitem{E1991}
J. A. Ellis, {Embedding rectangular grids into square grids}. IEEE Trans. Comput., 40:1(1991), 46–52.

\bibitem{OS2000}%
D. Sotteau and J. Opatrny,  
{Embeddings of complete binary trees into grids and extended grids with total vertex-congestion 1}, 
Discrete Applied Mathematics, 98(2000), 237-254. 

\bibitem{DWNS04}
B. Doina, W. B. Wolfgang, B. Natasa, L. Shahram, 
{An optimal embedding of honeycomb networks into hypercubes}, 
Parallel Processing Letters, 14(2004), 367-375.

\bibitem{FJ2007}
J. Fan J and X. Jia, {Embedding meshes into crossed cubes}, Information Sciences, 177:15(2007), 3151–3160.

\bibitem{LSAD2021}%ref
Jia-Bao Liu, Arul Jeya Shalini, Micheal Arockiaraj,and J. Nancy Delaila, Characterization of the Congestion Lemma on
Layout Computation, Journal of Mathematics, Article ID 2984703(2021), https://doi.org/10.1155/2021/2984703.

%hypercube is important
\bibitem{Xu2001}
Junming Xu, {Topological Structure and Analysis of Interconnection Networks}, Kluwer Academic Publishers, 2001, pp 105.

%hypercube work%%%%%%%%%%%%%\cite{Chen1988,PM2009,PM2011,RARM2012}

\bibitem{Chen1988}
B. Chen, {On embedding rectangular grids in hypercubes}, IEEE Trans. Comput., 37:10(1988), 1285–1288.

\bibitem{PM2009}
P. Manuel, I. Rajasingh, B. Rajan and H. Mercy, 
{Exact wirelength of hypercube on a grid},
Discrete Applied Mathematics, 157:7(2009), 1486--1495.

%\bibitem{PM2011}

\bibitem{RARM2012}%ref
Indra Rajasingh, Micheal Arockiaraj, Bharati Rajan and Paul Manuel, 
{Minimum wirelength of hypercubes into $n-$dimensional grid networks}, Information Processing Letters, 112(2012), 583-586.

%discuss
\bibitem{RRPR2014}
R.S.Rajan, I. Rajasingh, N.Parthiban,  T.M. Rajalaxmi, 
A linear time algorithm for embedding hypercube into cylinder and torus, Theoretical Computer Science, 542(2014),108–115.

\bibitem{LT2021}
Qinghui Liu and Zhiyi Tang, {A rigorous proof on circular wirelength for hypercubes}, Acta Mathematica Scientia,  43B:2(2023), 919–941.

\bibitem{Tang2022}
Zhiyi Tang, Optimal embedding of hypercube into cylinder,
Theoretical Computer Science, 923(2022), 327–336.

%EIP
\bibitem{H2004}
L. H. Harper, {Global methods for isoperimetric problems}, Cambridge university press, 2004, Chapter 1.
	
%label a number:{PM2009,PM2011,RRPR2014,AS2015,LT2021,Tang2022}

%hypbercube notation
\bibitem{RR2022}
C. Rashtchian, W. Raynaud, {Edge isoperimetric inequalities for powers of the hypercube}, 2022, latest arXiv:1909.10435.	
	

%\bibitem{JQ2015}
%Weixing Ji, Qinghui Liu, Guizhen Wang and Zhuojia Shen, {Embedding of hypercube into cylinder}, 2015, arXiv:1511.07932v1.

%\bibitem{CYWC2004}%ref
%Jou-Ming Chang, Jinn-Shyong Yang, Yue-Li Wang and Yuwen Cheng, Panconnectivity, fault-tolerant hamiltonicity and hamiltonian-connectivity in alternating group graphs, Networks, 44(4)(2004),302-310.

%\bibitem{FLJL2005}%ref
%Jianxi Fan, Xiaola Lin, Xiaohua Jia and Rynson W. H. Lau, Edge-Pancyclicity of Twisted Cubes, ISAAC2005, Lecture Notes in Computer Science, vol 3827.



\end{thebibliography}
\end{document}